\documentclass[a4paper,10pt]{amsart}
\usepackage{epsfig}
\usepackage{graphicx,color}
\usepackage{amsmath,amsfonts,amssymb,eufrak}
\usepackage{float}
\usepackage{graphics}
\usepackage{latexsym}
\usepackage{float}
\usepackage{verbatim}
\usepackage{xy}
\xyoption{matrix}
\xyoption{arrow}
\input xy 
\usepackage{rotating}
\xyoption{all}         

\newtheorem{lem}{Lemma}[section]
\newtheorem{prop}[lem]{Proposition}
\newtheorem{cor}[lem]{Corollary}
\newtheorem{thm}[lem]{Theorem}

\newcommand{\Ext}{\operatorname{Ext}\nolimits}

\newcommand{\mo}{\operatorname{mod}\nolimits}

\newcommand{\End}{\operatorname{End}\nolimits}
\newcommand{\add}{\operatorname{add}\nolimits}

\begin{document}
\title{Coloured quivers of type $A$ and the cell-growth problem} 
\author{Hermund Andr\' e Torkildsen} 

\begin{abstract} We use a geometric description of $m$-cluster
  categories of Dynkin type $A$ to count the the number of coloured
  quivers in the $m$-mutation class of quivers of Dynkin type
  $A$. This is related to angulations of polygons and the cell-growth
  problem.  
\end{abstract}

\maketitle


\section*{Introduction}

Quiver mutation \cite{fz1} induces an equivalence relation on the
set of quivers, and the mutation class of a quiver $Q$ consists of all
quivers mutation equivalent to $Q$. It was shown in \cite{br} that the
mutation class of an acyclic quiver $Q$ is finite if and only if 
the underlying graph of $Q$ is either Dynkin, extended Dynkin or has
at most two vertices. In \cite{to2} an analogous result was obtained
for coloured quivers. Let $Q$ be a coloured quiver obtained from an
$m$-cluster tilting object. Then $Q$ has finite mutation class if and
only if $Q$ is mutation equivalent to a quiver $Q'$, where the Gabriel
quiver of $Q'$ is Dynkin, extended Dynkin or has at most two vertices,
and $Q'$ has only arrows of colour $0$ and $m$.

For finite mutation classes, it is a natural question to ask how many
non-isomorphic quivers there are in the class. In \cite{to1,bto,brs}
explicit formulas were given to compute the number of non-isomorphic
quivers in the mutation class of quivers of type $A$, $D$ and
$\widetilde{A}$ respectively, and in this paper we will give an
explicit formula for the number of non-isomorphic coloured quivers of
type $A$. This generalizes a result in \cite{to1}. 

In \cite{bm1,bm2} the authors gave geometric descriptions of the
$m$-cluster categories of type $D$ and $A$. This generalizes the
geometric descriptions of cluster categories given in
\cite{ccs,s}. In this paper we will use this to investigate mutation
classes arising from $m$-cluster categories of type $A$. 

There is a $1-1$ correspondence between $m$-cluster tilting objects
and $(m+2)$-angulations of an $(nm+2)$-gon in this case. This enables
us to count the number of $m$-cluster tilting objects in the
$m$-cluster category. Furthermore we show that there is a $1-1$
correspondence between coloured quivers of $m$-cluster tilted algebras
and $(m+2)$-angulations of $(nm+2)$-gons, where two
$(m+2)$-angulations are considered equivalent if they are rotations of
each other. Counting such $(m+2)$-angulations is related to the
cell-growth problem investigated by for example Harary, Palmer and
Read in \cite{hpr}, and they provide the generating functions that we
need.


\section{$m$-cluster categories}

Cluster categories were defined in \cite{bmrrt} in the general case and
in \cite{ccs} in the $A$-case as a categorical model of the
combinatorics of cluster algebras. Some cluster categories have a nice
geometric description in terms of triangulations of certain polygons,
see \cite{ccs, s}. 

Let $H=kQ$ be a finite dimensional hereditary algebra over an
algebraically closed field $k$, where $Q$ is a quiver with $n$
vertices. Let $\mathcal{D}^b(H)$ be the bounded derived category of
$\mo H$. The orbit category $\mathcal{C}_H =
\mathcal{D}^b(H)/\tau^{-1}[1]$, where $\tau$ is the Auslander-Reiten
translation and $[1]$ is the shift functor, is called the cluster
category of $H$. In \cite{ccs} this category was defined as a category
of diagonals of a regular $(n+3)$-gon. The objects are direct sums of
diagonals and the morphisms are spanned by elementary moves modulo the
mesh-relations. For the $D$ case \cite{s} considered $n$-gons with a
puncture. Schiffler defined a category with direct sums of diagonals
(or tagged edges) as objects. The morphism space is spanned by
elementary moves modulo the mesh relations. He showed that this
category was equivalent to the cluster category of type $D$.

A generalization of cluster categories are the $m$-cluster
categories. Consider the orbit category $\mathcal{C}_{H}^{m} =
\mathcal{D}^b(H)/\tau^{-1}[m]$, for some positive integer $m$. This
category is called the $m$-cluster category of $H$. It has been
investigated in several papers, see for example
\cite{bm1,bm2,bt,iy,k,t,w,z,zz}. The $m$-cluster category is a  
Krull-Schmidt category for all $m$, and it has an AR-translate 
$\tau$. From \cite{k} we also know that it is a triangulated category
for all $m$. The indecomposable objects in $\mathcal{C}_H^m$ are of
the form $X[i]$, with $0 \leq i < m$, where $X$ is an indecomposable 
$H$-module, or of the form $P[m]$, where $P$ is a projective
$H$-module. 

Baur and Marsh considered in \cite{bm1,bm2} $m$-cluster categories,
and they generalized \cite{s,ccs} for arbitrary $m$'s. In this paper
we will consider their generalization of geometric descriptions of
$m$-cluster categories of type $A$.

If $T$ is an object in $\mathcal{C}_H^m$ with the property that $X$ is
in $\add T$ if and only if $\Ext_{\mathcal{C}_H^m}^i(T,X)=0$ for all
$i \in \{1,2,...,m\}$, then $T$ is called an $m$-cluster tilting
object. An object $X$ is called maximal $m$-rigid if it has the
property that $X \in \add T$ if and only if
$\Ext_{\mathcal{C}_H^m}^i(T \oplus X,T \oplus X)=0$ for all $i \in
\{1,2,...,m\}$. In \cite{w,zz} it was shown that an object which is
maximal $m$-rigid is also an $m$-cluster tilting object, and that an
$m$-cluster tilting object $T$ always has $n$ non-isomorphic
indecomposable summands \cite{z}. 

If $T$ is an $m$-cluster tilting object in $\mathcal{C}_H^m$, the
algebra $\End_{\mathcal{C}_H^m}(T)$ is called an $m$-cluster tilted
algebra. In \cite{to1} it was shown that $1$-cluster tilted algebras
of type $A_n$ (up to isomorphism) are in $1-1$ correspondence with
triangulations of $(n+3)$-gons, where two triangluations are
considered equivalent if they are rotations of each other. A similar
result was obtained in \cite{bto} for type $D$. We will see that this
does not hold for an arbitrary $m$.


\section{Quiver mutation}

Let $\bar{T}$ be an object in $\mathcal{C}_H^m$ with $n-1$
non-isomorphic indecomposable direct summands such that 
$\Ext_{\mathcal{C}_H^m}^i(\bar{T},\bar{T})=0$ for $i \in
\{1,2,...,m\}$. Such an object is called an almost complete
$m$-cluster tilting object, and in \cite{w,zz} they show that
$\bar{T}$ always has exactly $m+1$ complements.

Let $T_k^{(c)}$, where $c\in \{0,1,2,...,m\}$, be the complements
of $\bar{T} = T/T_k$. Then the complements are connected by $m+1$ exchange
triangles \cite{iy} 

$$T_k^{(c)} \rightarrow B_k^{(c)} \rightarrow T_k^{(c+1)}\rightarrow,$$
where $B_k^{(c)}$ is in $\add{\bar{T}}$.

Let $T$ be an $m$-cluster tilting object in $\mathcal{C}_H^m$. In
\cite{bt} the authors associate to $T$ a quiver $Q_T$ in the following
way. There is a vertex in $Q_T$ for every indecomposable summand of
$T$, and the arrows have colours chosen from the set
$\{0,1,2,...,m\}$. If $T_i$ and $T_j$ are two indecomposable summands of
$T$ corresponding to vertex $i$ and $j$ in $Q_T$, there are $r$ arrows
from $i$ to $j$ of colour $c$, where $r$ is the multiplicity of $T_j$
in $B_i^{(c)}$.

They show that quivers obtained in this way have the following
properties.
\begin{enumerate}
\item There are no loops.
\item If there is an arrow from $i$ to $j$ with colour $c$, then there exists no arrow
from $i$ to $j$ with colour $c' \neq c$.
\item If there are $r$ arrows from $i$ to $j$ of colour $c$, then there are $r$
arrows from $j$ to $i$ of colour $m-c$.
\end{enumerate}

Coloured quiver mutation keeps track of the exchange of indecomposable summands
of an $m$-cluster tilting object. The mutation of $Q_T$ at vertex $j$
is a quiver $\mu_j(Q_T)$ obtained as follows.  

\begin{enumerate}
\item For each pair of
  arrows $$\xymatrix{i\ar[r]^{(c)}&j\ar[r]^{(0)}&k\\}$$ where $i \neq
  k$ and $c \in \{0,1,...,m\}$, add an arrow from $i$ to $k$ of colour
  $c$ and an arrow from $k$ to $i$ of colour $m-c$. 
\item If there exist arrows of different colours from a vertex $i$ to
  a vertex $k$, cancel the same number of arrows of each colour until
  there are only arrows of the same colour from $i$ to $k$. 
\item Add one to the colour of all arrows that goes into $j$, and
  subtract one from the colour of all arrows going out of $j$.
\end{enumerate}

%

In \cite{bt} the following theorem is proved.

\begin{thm}
Let $T = \oplus_{i=1}^{n} T_i$ be an $m$-cluster tilting object in
$\mathcal{C}_H^m$. Let $T' = T / T_j \oplus T_j^{(1)}$ be an
$m$-cluster tilting object where there is an exchange triangle $$T_j
\rightarrow B_j^{(0)} \rightarrow T_j^{(1)} \rightarrow.$$ Then
$Q_{T'} = \mu_j(Q_T)$. 
\end{thm}

Let us denote by $Q_G$, where $Q=Q_T$ is a coloured quiver and $T$ an
$m$-cluster tilting object, the quiver obtained from $Q$ by removing
all arrows with colour different from $0$. Then the $m$-cluster tilted
algebra $\End_{\mathcal{C}_H^m}(T)$ has Gabriel quiver $Q_G$. From
\cite{zz} we have the following proposition and corollary. See also
\cite{bt}. 

\begin{prop}\label{tiltreached}
Any $m$-cluster tilting object can be reached from any other
$m$-cluster tilting object via iterated mutation.
\end{prop}

\begin{cor}
For an $m$-cluster category $\mathcal{C}_H^m$ of the acyclic quiver
$Q$, all quivers of $m$-cluster tilted algebras are given by repeated 
mutation of $Q$.
\end{cor}

For $m=1$, it was shown in \cite{br} that the mutation class of an
acyclic quiver $Q$ if finite if and only if the underlying graph of
$Q$ is either Dynkin, extended Dynkin or $Q$ has at most two
vertices. This was generalized in \cite{to2}, and we have the
following theorem and corollary.

\begin{thm}
Let $k$ be an algebraically closed field and $Q$ a connected finite
quiver without oriented cycles. The following are equivalent for $H =
kQ$.
\begin{enumerate}
\item There are only a finite number of basic $m$-cluster tilted algebras
  associated with $H$, up to isomorphism. 
\item There are only a finite number of Gabriel quivers occurring for
  $m$-cluster tilted algebras associated with $H$, up to isomorphism.
\item $H$ is of finite or tame representation type, or has at most two
  non-isomorphic simple modules.
\item There are only a finite number of $\tau$-orbits of $m$-cluster
  tilting objects associated with $H$. 
\item There are only a finite number of coloured quivers occurring for
  $m$-cluster tilting objects associated with $H$, up to isomorphism.
\item The mutation class of a coloured quiver arising from an
  $m$-cluster tilting object associated with $H$, is finite. 
\end{enumerate}
\end{thm}

\begin{cor}
A coloured quiver $Q$ corresponding to an $m$-cluster tilting object,
has finite mutation class if and only if $Q$ is mutation equivalent to
a quiver $Q'$, where $Q'_G$ has underlying graph Dynkin or extended
Dynkin, or it has at most two vertices, and there are only arrows of
colour $0$ and $m$ in $Q'$.
\end{cor}

\section{Category of $m$-diagonals of a regular $(nm+2)$-gon}

From now on we want to consider only the $A$ case. Baur and Marsh 
gave in \cite{bm2} a geometric description of the $m$-cluster category
in the $A$ case, and we will use this to investigate the mutation class
of coloured quivers.

Let $n$ and $m$ be positive integers and $P_{n,m}$ a regular
$(nm+2)$-gon. Label the vertices on the border of the polygon
clockwise from $1$ to $nm+2$. See Figure \ref{figpolygon}. 

  \begin{figure}[htp]
  \begin{center}
    \includegraphics[width=4.5cm]{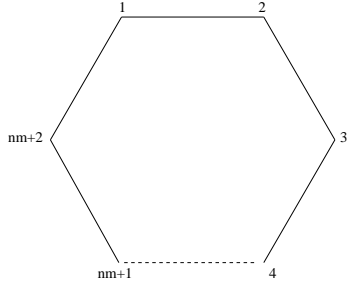}
  \end{center}\caption{\label{figpolygon} The $(nm+2)$-gon $P_{n,m}$.}
  \end{figure}

An $m$-diagonal $\alpha$ is a straight line
between two vertices on the border such that $\alpha$ divides the
polygon into two parts, each with number of sides congruent to $2$
modulo $m$. We denote by $\alpha_{i,j}$ the $m$-diagonal between
the vertex $i$ and $j$.

We say that a set $\Delta$ of $m$-diagonals do not cross if no two
diagonals in $\Delta$ cross at the interior of the polygon. A maximal
set of $m$-diagonals that do not cross consists of exactly $n-1$
diagonals, and they divide the polygon into $(m+2)$-gons. We call such
a set an $(m+2)$-angulation of $P_{n,m}$. See \cite{bm2}. We
note that if $m=1$, a maximal set of diagonals is a triangulation of
the regular polygon $P_{n,1}$, and it consists of $n-1$ diagonals. See
Figure \ref{figexampleangulations} for some examples.

  \begin{figure}[htp]
  \begin{center}
    \includegraphics[width=4.1cm]{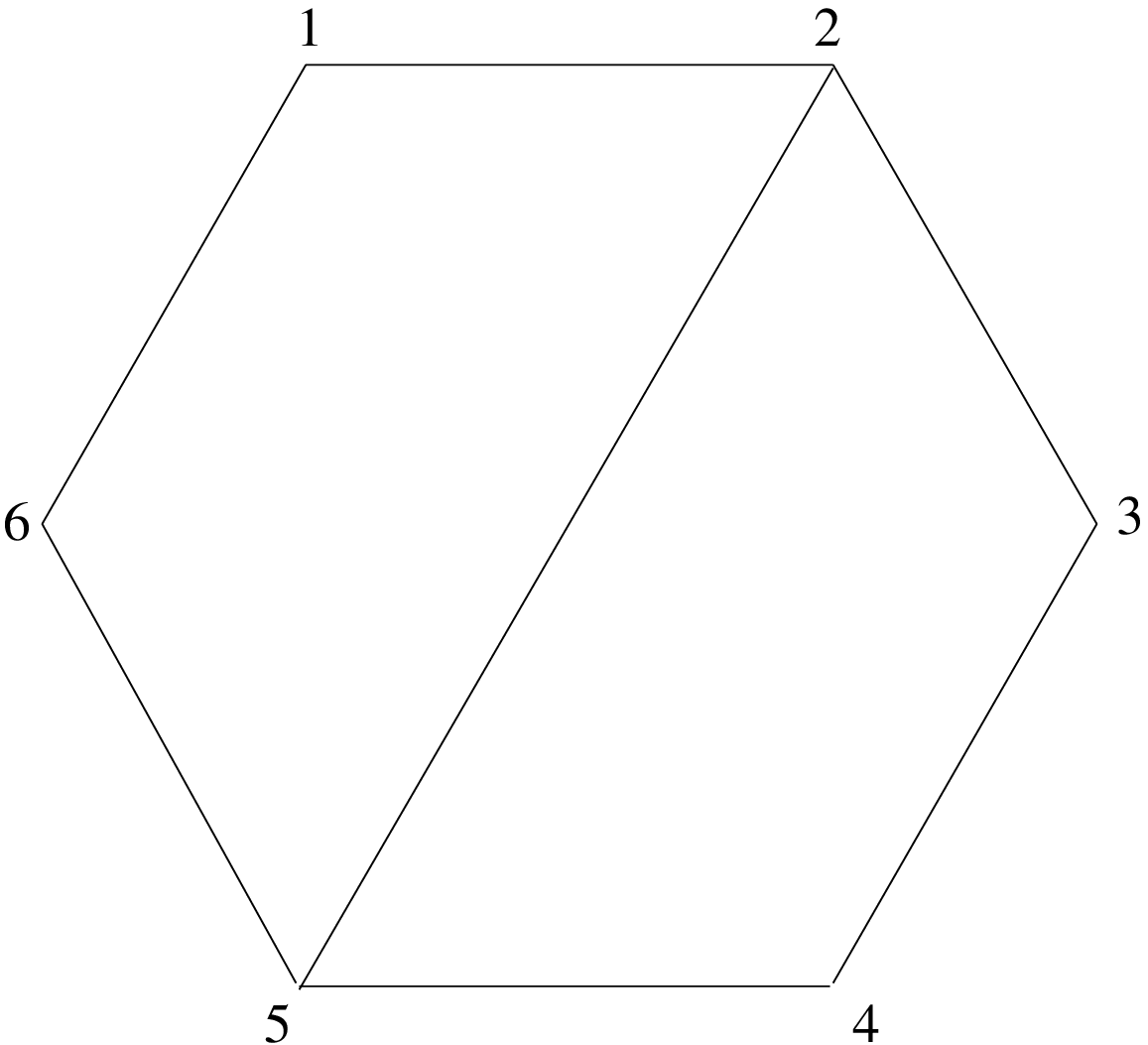}
    \includegraphics[width=3.75cm]{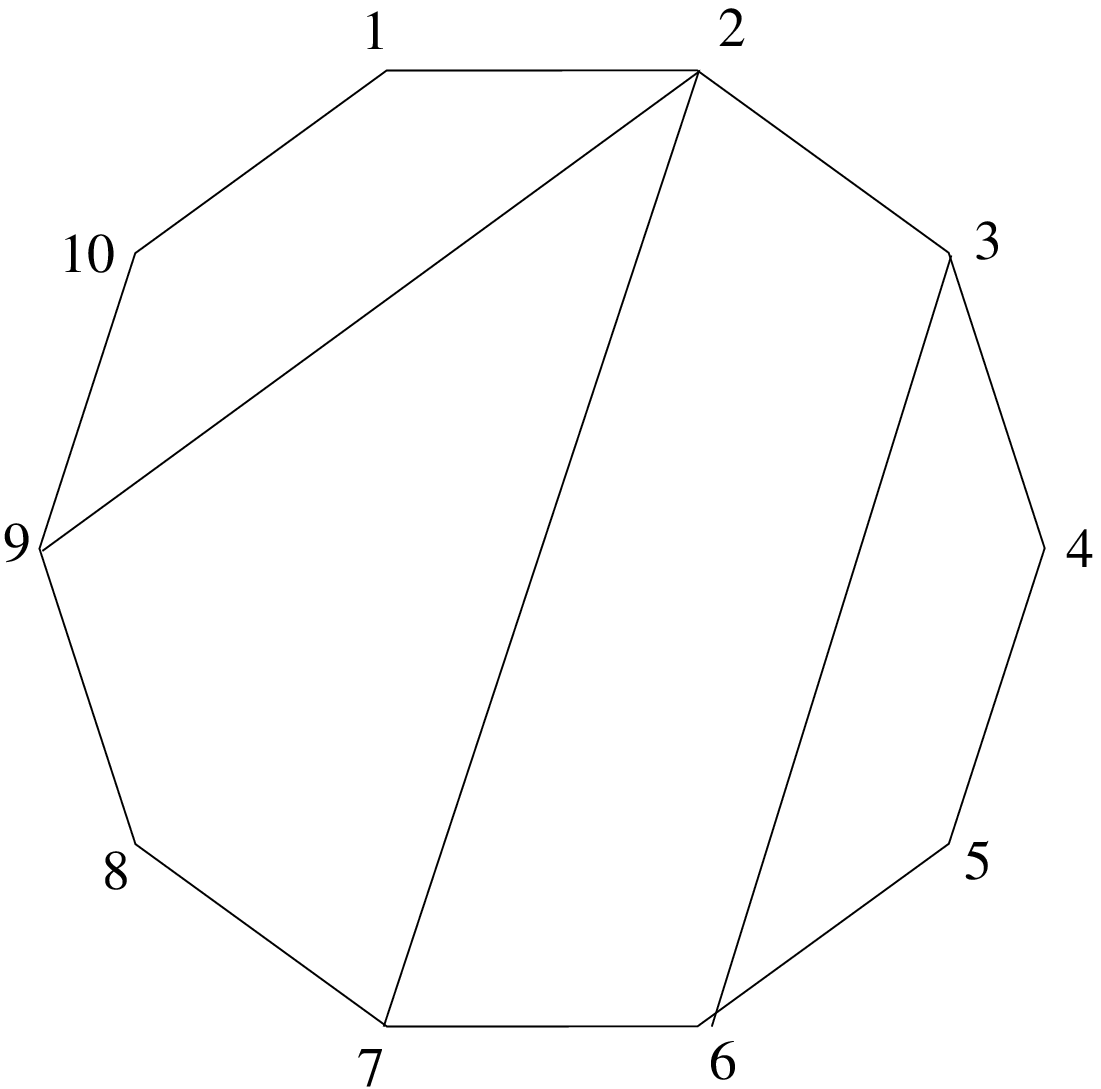}
  \end{center}\caption{\label{figexampleangulations}A $4$-angulation
      of $P_{2,2}$ and a $4$-angulation of $P_{4,2}$.}  
  \end{figure}

Given an $(m+2)$-angulation $\Delta$ of $P_{n,m}$, \cite{bt} define an 
operation on the diagonals. Let $\alpha$ be an $m$-diagonal in
$\Delta$, and consider the set of diagonals $(\Delta \backslash
\alpha)$. By removing $\alpha$, we obtain a $(2m+2)$-gon, and there
are $m+1$ possibilities to put it back to obtain a new
$(m+2)$-angulation. We call $\alpha$ a diameter of the inner
$(2m+2)$-gon, because it geometrically connects two opposite vertices
in the $(2m+2)$-gon. If we replace $\alpha$ with another diagonal, say
$\beta$, then $\beta$ is also a diameter of the $(2m+2)$-gon, and
hence all possibilities to replace $\alpha$ with another diagonal
correspond to rotating the $(2m+2)$-gon. The operation is defined as
follows. For any $m$-diagonal $\alpha$, define the mutation at
$\alpha$ to be the new $(m+2)$-angulation obtained by rotating the
$2m+2$-gon corresponding to $\alpha$ clockwise. See Figure
\ref{figexamplemutations}.

  \begin{figure}[htp]
  \begin{center}
    \includegraphics[width=9cm]{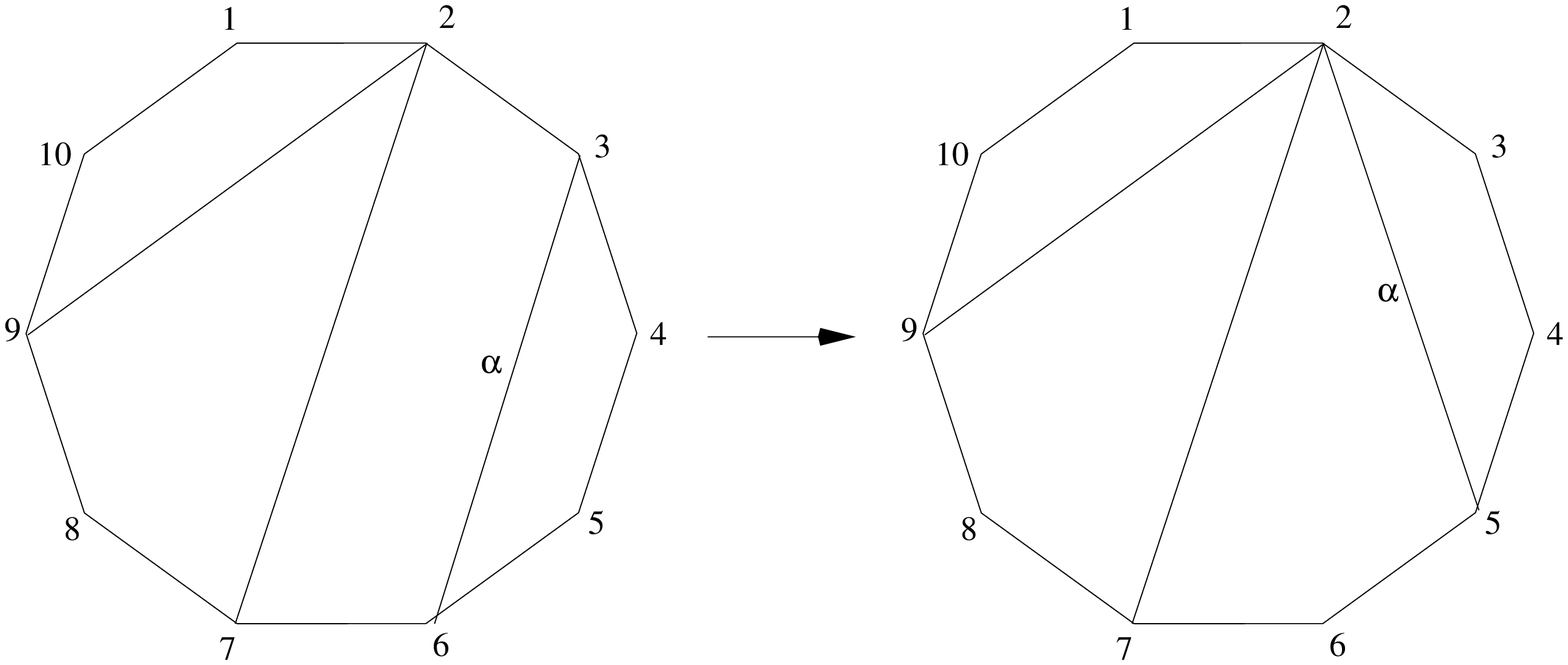}
  \end{center}\caption{\label{figexamplemutations}Mutation at the
    diagonal $\alpha$.}  
  \end{figure}

For an $(m+2)$-angulation $\Delta$ of $P_{n,m}$, we can define a
coloured quiver with $n-1$ vertices in the following way (see
\cite{bt}). The vertices are the diagonals in $\Delta$ and there is an
arrow from $\alpha$ to $\beta$ if they both lie in some $(m+2)$-gon in
$\Delta$. The colour of the arrow is the number of edges forming the
segment of the boundary of the $(m+2)$-gon which lies between $\alpha$
and $\beta$, counterclockwise from $\alpha$. See Proposition 11.1 in
\cite{bt}. We write $Q_{\Delta}$ for the coloured quiver obtained from
$\Delta$.

It is known that the quiver obtained in this way is a coloured quiver
of type $A_{n-1}$. It is also known that all quivers of this type can
be obtained by such an $(m+2)$-angulation. Also, quiver mutation
commutes with mutation at diagonals. Figure
\ref{figangulationandquiver} shows three $4$-angulations of $P_{4,2}$
and the corresponding coloured quivers. The Figure also illustrates
mutation at the middle diagonal.  

  \begin{figure}[htp]
  \begin{center}
    \includegraphics[width=12.6cm]{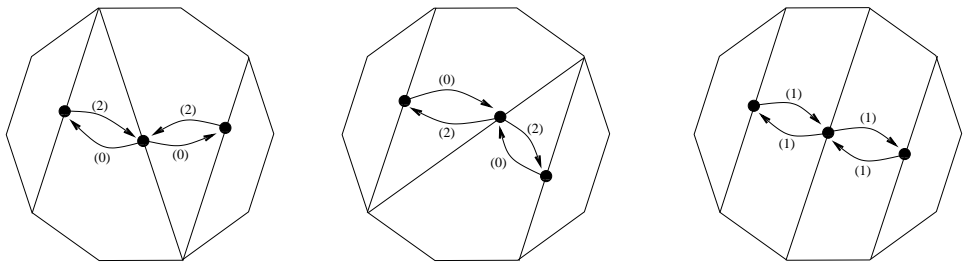}
  \end{center}\caption{\label{figangulationandquiver} Examples of 
    $4$-angulations of $P_{4,2}$ and the corresponding quivers.}  
  \end{figure}

Let $\alpha$ be an $m$-diagonal. Denote by $r \alpha$ the $m$-diagonal
obtained by rotating the polygon one step in the counterclockwise
direction. Also we let $r \Delta$ be the $(m+2)$-angulation obtained
by applying $r$ on each diagonal in $\Delta$. It is clear that $r$
preserves the corresponding coloured quiver. 

Baur and Marsh define in \cite{bm2} a category $\mathcal{C}_{n,m}$ of
diagonals. This is a generalization of \cite{ccs}, where they consider
the case $m=1$. The indecomposable objects are the $m$-diagonals in
this additive category. The morphisms are generated by certain
elementary moves of $m$-diagonals modulo mesh-relations. There is an
elementary move $\alpha_{i,j} \rightarrow \alpha_{i,k}$ if both are
$m$-diagonals such that they form an $(m+2)$-gon with the boundary, and
we can obtain $\alpha_{i,k}$ by rotating $\alpha_{i,j}$ clockwise
about the vertex $i$. Let $\alpha$ be an $m$-diagonal, and consider
all $m$-diagonals $\beta_1, \beta_2,...,\beta_t$ with an elementary
move to $\alpha$, say $f_i: \beta_i \rightarrow \alpha$. Then there
exist elementary moves $g_i:r^m \alpha \rightarrow \beta_i$. The
mesh-relation at $\alpha$ is defined to be $$\sum_{i=1}^{t}f_i g_i.$$ 

Baur and Marsh show that this category is equivalent to the
$m$-cluster category of Dynkin type $A_{n-1}$. They also define a quiver
associated to the category, where the vertices are the $m$-diagonals and
the arrows are the elementary moves. This quiver is isomorphic to the
AR-quiver of the $m$-cluster category. The AR-translation is denoted
by $\tau$, and it is in fact equal to $r^m$. See Figure
\ref{figarquiverex} for the AR-quiver of $\mathcal{C}_{4,2}$. 

  \begin{figure}[htp]
  \begin{center}
    \includegraphics[width=12.6cm]{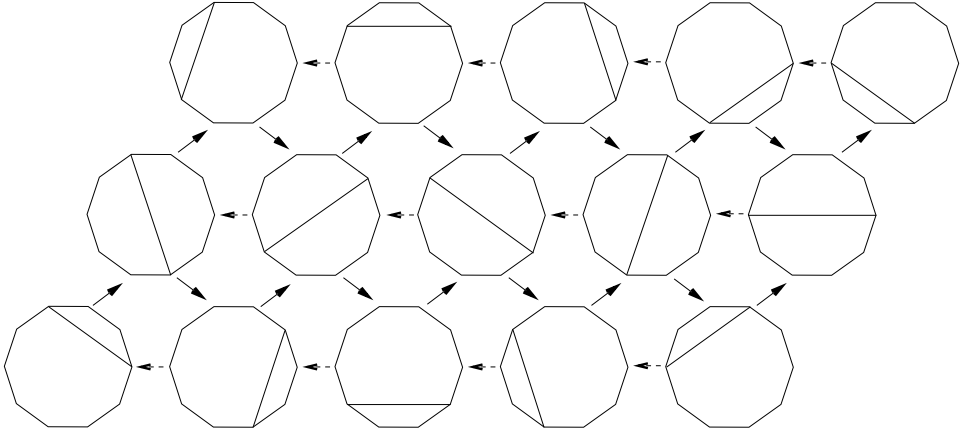}
  \end{center}\caption{\label{figarquiverex} AR-quiver of
    $\mathcal{C}_{4,2}$.}   
  \end{figure}

Furthermore, we have that if $\Delta$ is an $(m+2)$-angulation, then
$\Delta$, as a direct sum of its $m$-diagonals, is a basic
$m$-cluster-tilting object. Its endomorphism ring has quiver
$Q_{\Delta}$.

\section{Bijection between the mutation class of coloured quivers of
  type $A_{n-1}$ and $m$-angulations of polygons}

In this section we want to generalize the results in \cite{to1}. The
proof of the main theorem uses the same ideas as in
\cite{to1}. However, the situation is more complicated, and there are
more cases to consider. We prepare for the main theorem with a series
of lemmas.

Let $\mathcal{T}_{n,m}$ be the set of $(m+2)$-angulations of
$P_{n,m}$, and denote by $\mathcal{M}_{n,m}$ the mutation class of
coloured quivers of type $A_{n}$. From the previous section we have a 
surjective map $$\sigma_{n,m}: \mathcal{T}_{n,m} \rightarrow
\mathcal{M}_{n-1,m},$$ where $\sigma_{n,m}(\Delta) = Q_{\Delta}$. This
function is not injective, since for example $\sigma_{n,m}(\Delta) =
\sigma_{n,m}(r\Delta)$.   

In the same way as in \cite{to1} we define an equivalence relation on
$\mathcal{T}_{n,m}$ by letting $\Delta \sim \Delta'$ if $\Delta' =
r^i \Delta$ for some $i$. Recall that $r^m = \tau$. In \cite{to2} it
was observed that if $T$ is an $m$-cluster tilting object in
$\mathcal{C}_H^m$, then $Q_T$ is isomorphic to $Q_{T[i]}$ for all
$i \in \mathbb{Z}$. Since $[i]$ is an equivalence on the $m$-cluster
category, we have that $\End_{\mathcal{C}}(T)
\simeq \End_{\mathcal{C}}(T[i])$ for all $i \in
\mathbb{Z}$. Considering the category of $m$-diagonals, we have the
following proposition. The proof is left to the reader, and it is
straightforward from the AR-quiver.

\begin{prop} Let $\mathcal{C}_{n,m}$ be the category of $m$-diagonals
  of $P_{n,m}$, then $r = [1]$. 
\end{prop}

We want to show that the function 

$$\widetilde{\sigma}_{n,m}: (\mathcal{T}_{n,m}/\!\!\sim) \rightarrow
\mathcal{M}_{n-1,m},$$
induced from $\sigma_{n,m}$, is bijective. For $m=1$ this is already
known for all $n$.

We say that an $m$-diagonal $\alpha$ in an $(m+2)$-angulation $\Delta$
is close to the border in $\Delta$ if $\alpha$ lies in an $(m+2)$-gon
together with edges only on the border. If $\Delta$ is an
$(m+2)$-angulation and $\alpha$ is close to the border, we define
$\Delta / \alpha$ to be the $(m+2)$-angulation of $P_{n-1,m}$ obtained
from $\Delta$ by letting $\alpha$ be a border edge and leaving all
other diagonals unchanged. See Figure \ref{figfactoringout}. We say
that we factor out $\alpha$. 

For a diagonal $\alpha$ in some $(m+2)$-angulation $\Delta$,
we always denote by $v_{\alpha}$ the corresponding vertex in
$Q_{\Delta}$. We generalize the proofs in \cite{to1} to obtain the
following lemmas.

\begin{lem}\label{lemcommutes}
If $\Delta$ is an $(m+2)$-angulation and $\alpha$ close to the border
in $\Delta$, then $Q_{\Delta}/v_{\alpha}$ is
connected. Furthermore, factoring out $v_{\alpha}$ corresponds to
factoring out $\alpha$ in $\Delta$, i.e. $Q_{\Delta / \alpha} =
Q_{\Delta}/ v_{\alpha}$.  
\end{lem}
\begin{proof} Factoring out $\alpha$ in $\Delta$ corresponds to
  removing the 
  vertex $v_{\alpha}$ in $Q_{\Delta}$ and all arrows adjacent to
  $v_{\alpha}$. All the other $m$-diagonals remain unchanged, and
  hence all arrows between all other vertices remain unchanged. It
  follows that $Q_{\Delta / \alpha} = Q_{\Delta}/ v_{\alpha}$. Then it
  is also easy to see that the resulting quiver is connected, since we
  obtain a new $(m+2)$-angulation.   
\end{proof}

\begin{lem}\label{lemconnected} If $\Delta$ is an $(m+2)$-angulation
  and $Q_{\Delta}$ the corresponding coloured quiver, then the
  coloured quiver $Q_{\Delta}/v_{\alpha}$, obtained from $Q_{\Delta}$
  by factoring out a vertex $v_{\alpha}$, is connected if and only if
  the corresponding $m$-diagonal $\alpha$ is close to the border.    
\end{lem}
\begin{proof} If $\alpha$ is not close to the border, $\alpha$ divides
  $\Delta$ into two parts $A$ and $B$ with at least one $m$-diagonal
  in each part. Let $\beta$ be an $m$-diagonal in $A$ and $\beta'$ an
  $m$-diagonal in $B$. If $\beta$ and $\beta'$ would determine a
  common $(m+2)$-gon in $\Delta$, the third $m$-diagonal would cross
  $\alpha$. Hence there is no path between the subquiver determined by
  $A$ and the subquiver determined by $B$, except those passing
  through $v_{\alpha}$.
  
  If $\alpha$ is close to the border, then, by Lemma
  \ref{lemcommutes}, factoring out $v_{\alpha}$ does not disconnect
  the quiver.  
\end{proof}

  \begin{figure}[htp]
  \begin{center}
    \includegraphics[width=7.5cm]{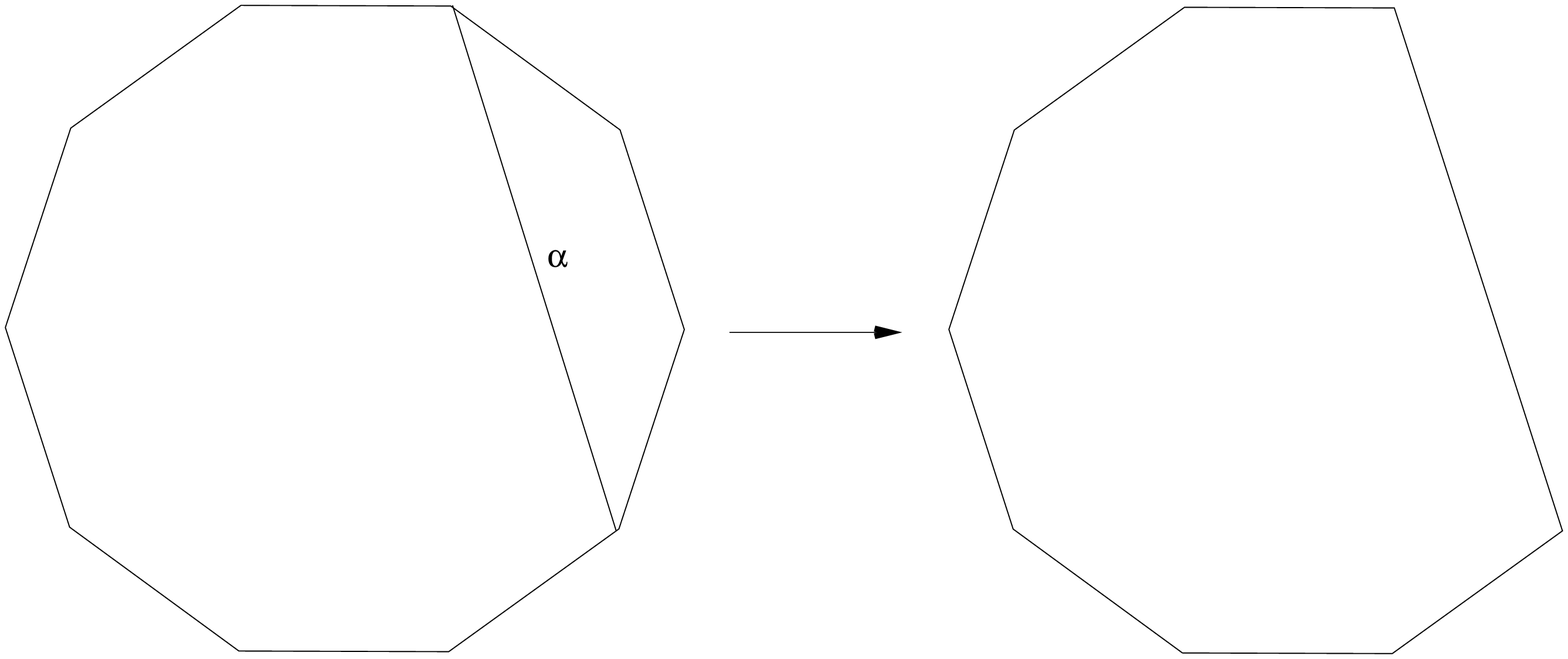}
  \end{center}\caption{\label{figfactoringout} Factoring out a
    diagonal close to the border, where $m=2$ and $n=4$.}  
  \end{figure}

We have the following proposition.

\begin{prop} Let $Q$ be a coloured quiver of an $m$-cluster-tilted
  algebra of type $A_{n-1}$. If $Q / v$ is the coloured quiver
  obtained from $Q$ by factoring out a vertex $v$ such that $Q / v$
  is connected, then $Q / v$ is a coloured quiver of an
  $m$-cluster-tilted algebra of type $A_{n-2}$.
\end{prop}
\begin{proof} The vertex $v$ corresponds to an $m$-diagonal $\alpha$
  close to the border of an $(m+2)$-angulation $\Delta$ of $P_{n,m}$,
  by Lemma \ref{lemconnected}. Since $\Delta / \alpha$ is an
  $(m+2)$-angulation of $P_{n-1,m}$, the claim follows. 
\end{proof}

Next we want to define an extension of an $(m+2)$-angulation $\Delta$
of $P_{n,m}$, by adding $m$ new vertices to the border of the polygon
and an $m$-diagonal $\alpha$ such that $\alpha$ is close to the border
and $\Delta \cup \alpha$ is an $(m+2)$-angulation of $P_{n+1,m}$. More 
precisely, consider any border edge $e$ between vertex $v_i$ and
$v_{i+1}$ on $P_{n,m}$, and extend the polygon with $m$ new vertices
between $v_i$ and $v_{i+1}$, making $e$ the new diagonal close to the
border in $P_{n+1,m}$. We denote by $\Delta(e)$ the new
$(m+2)$-angulation thus obtained. See Figure
\ref{figextendingate}. Obviously $\Delta(e)/e = \Delta$. Of course, we
can do this at any border edge in $\Delta$. 

  \begin{figure}[htp]
  \begin{center}
    \includegraphics[width=10cm]{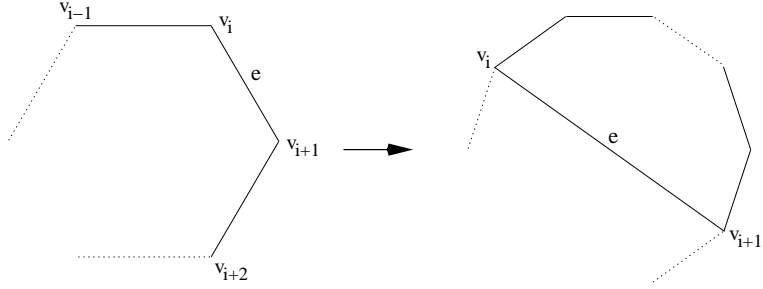}
  \end{center}\caption{\label{figextendingate} Extending an
    $(m+2)$-angulation of $P_{n,m}$ at the border edge $e$, making $e$
    a diagonal close to the border in an $(m+2)$-angulation $\Delta(e)$
    of $P_{n+1,m}$.}    
  \end{figure}

Let $\alpha$ be some $m$-diagonal in an $(m+2)$-angulation
$\Delta$. We say that an $m$-diagonal or border edge $e$ is adjacent
to $\alpha$ if $e$ and $\alpha$ both lie in some common $(m+2)$-gon in
$\Delta$. If $e$ is an $m$-diagonal, then there is some coloured arrow
from $v_e$ to $v_{\alpha}$ in $Q_{\Delta}$ if and only if $\alpha$ and
$e$ are adjacent. If $e$ is a border edge of the polygon, then
$Q_{\Delta(e)}$ has a coloured arrow from $v_e$ to $v_{\alpha}$ if and
only if $e$ and $\alpha$ are adjacent.  

  \begin{figure}[htp]
  \begin{center}
    \includegraphics[width=4cm]{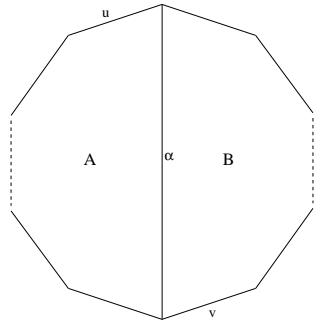}
  \end{center}\caption{\label{figadjacenttoalpha} Extending an
    $(m+2)$-angulation $\Delta$ of $P_{n,m}$ such that the vertex
    corresponding to the new diagonal in $Q_{\Delta(e)}$ is adjacent to
    the vertex corresponding to $\alpha$. See Lemma
    \ref{extensionnoniso}.} 
  \end{figure}

Suppose $\alpha$ is an $m$-diagonal in an $(m+2)$-angulation $\Delta$ of
$P_{n,m}$, and consider Figure \ref{figadjacenttoalpha}. Then $\alpha$
divides $\Delta$ into two parts $A$ and $B$. Since there can be
at most $(m+1)$ border edges adjacent to $\alpha$ in $A$ and $(m+1)$
border edges adjacent to $\alpha$ in $B$, there are at most $2(m+1)$
ways to extend $\Delta$ such that the vertex corresponding to the new
diagonal is adjacent to the vertex corresponding to $\alpha$. However,
by definiton of $Q_{\Delta}$, there is at most one border edge $e$ in
$A$ such that $Q_{\Delta(e)}$ has an arrow of a particular colour $c$
from $v_e$ to $v_{\alpha}$, and at most one border edge $e'$ in $B$
such that $v_{e'}$ has an arrow of colour $c$ to $v_{\alpha}$.

Before we prove the next lemma, recall that $r^i \Delta$ denotes the
rotation of $\Delta$ $i$ steps in the counterclockwise direction. In
the same way we write $r^i \alpha$, where $\alpha$ is an $m$-diagonal,
for the $m$-diagonal obtained from $\alpha$ by rotating the polygon
$i$ steps in the clockwise direction, i.e. $r^i \alpha =
\alpha[i]$. We extend this definition to a set of $m$-diagonals, so
let $A$ be a set of diagonals, then we define 
$$r^i A = \{r^i \alpha| \alpha \in A\}.$$ We denote by $Q_{\Delta_A}$ 
the subquiver of $Q_{\Delta}$ determined by the diagonals only in $A$,
i.e. the subquiver of $Q_{\Delta}$ obtained by factoring out all
vertices corresponding to $m$-diagonals not in $A$.

\begin{lem}\label{uniqueness}
  Suppose $\Delta$ is an $(m+2)$-angulation such that $Q_{\Delta'}
  \simeq Q_{\Delta}$ for some $(m+2)$-angulation $\Delta'$ implies
  $\Delta' = r^i \Delta$ for some integer $i$. 

  Let $\alpha$ be an $m$-diagonal in $\Delta$ that divides $\Delta$
  into two parts $A$ and $B$. Suppose there is an isomorphism
  $Q_{\Delta} \stackrel{\theta}{\rightarrow} Q_{\Delta'}$ that sends
  $v_{\alpha}$ to $v_{\alpha'}$. Then $\alpha'$ divides $\Delta'$ into
  two parts $A'$ and $B'$, such that $\alpha' = r^i \alpha$, $A' = r^i
  A$ and $B' = r^i B$ for some integer $i$.  
\end{lem}
\begin{proof}
  Suppose there are $s$ $m$-diagonals in A and $t$ $m$-diagonals in B.  

  If $s=t$, the claim is clear, since $\Delta$ and $\Delta'$ are
  rotation equivalent, and there are no other $m$-diagonal in
  $\Delta$, except $\alpha$, that divides $\Delta$ into two parts with
  the same number of $m$-diagonals. 

  Suppose $s \neq t$. First we make a general definition. Let $\Delta$
  be an $(m+2)$-angulation and $\alpha$ an $m$-diagonal in $\Delta$
  that divides $\Delta$ into two parts $A$ and $B$.  
  Define $E_{\Delta,\alpha}$ to be the finite set of $m$-diagonals
  $\{\alpha_k\}$ in $\Delta$ that divide $\Delta$ into two parts $A_k$
  and $B_k$ such that there exist some integer $j$ with $r^j A_k = A$.
  The set $E_{\Delta,\alpha}$ is non-empty because $\alpha \in
  E_{\Delta,\alpha}$. Clearly $Q_{\Delta_{A_k}} \simeq
  Q_{\Delta_{A}}$. 
  
  Now, consider the situation in the Lemma. Clearly $\alpha'$ divides
  $\Delta'$ into two parts $A'$ and $B'$ such that $Q_{\Delta_A}$ is
  isomorphic to $Q_{\Delta'_{A'}}$ and $Q_{\Delta_B}$ is isomorphic to
  $Q_{\Delta'_{B'}}$. Also, we know that there exist an integer $i$
  such that $r^i \alpha = \alpha'$, since both $\alpha$ and $\alpha'$
  divides $\Delta$ and $\Delta'$ into two parts with $s$ and $t$
  $m$-diagonals. Now, suppose for a contradiction that $r^i A \neq
  A'$.   

  Construct a new $(m+2)$-angulation $\Delta''$ by setting $\Delta''=
  (r^i \Delta - r^i A) \cup A'$. Then $Q_{\Delta} \simeq
  Q_{\Delta''}$, since $Q_{\Delta_{A \cup \alpha}} \simeq
  Q_{\Delta_{A' \cup \alpha'}}$. But we have that $|E_{\Delta,\alpha}|
  > |E_{\Delta'',\alpha}|$ by definition of $\Delta''$, so $\Delta$ is
  not rotation equivalent to $\Delta''$. This is a contradiction, so
  hence $r^i A = A'$. Similarly with $B$ and $B'$.

  \end{proof}

\begin{lem}\label{extensionnoniso} 
  Let $\Delta$ be an $(m+2)$-angulation such that if $Q_{\Delta'}$ is
  isomorphic to $Q_{\Delta}$ for some $(m+2)$-angulation $\Delta'$,
  then $\Delta' = r^i \Delta$ for some integer $i$. 

  Let $\Delta(e)$ and $\Delta(e')$, with $e \neq e'$, be two
  extensions of $\Delta$. Then $\Delta(e)= r^i \Delta(e')$ for some
  integer $i$ if and only if $Q_{\Delta(e)}$ is isomorphic to
  $Q_{\Delta(e')}$. 

\end{lem}
\begin{proof} If $\Delta(e) = r^i\Delta(e')$ for some integer $i$,
  then clearly $Q_{\Delta(e)}$ is isomorphic to $Q_{\Delta(e')}$.

  Suppose $Q_{\Delta(e)}$ is isomorphic to $Q_{\Delta(e')}$, and
  suppose that the new $m$-diagonal $e$ is adjacent to $\alpha$, say
  $v_e$ has an arrow of colour $c$ to $v_{\alpha}$ in
  $Q_{\Delta(e)}$. The diagonal $\alpha$ divides $\Delta$ into two  
  parts $A$ and $B$. Suppose there are $s$ $m$-diagonals in $A$ and
  $t$ $m$-diagonals in $B$. Without loss of generality we can assume
  that $e$ is a border edge in $A$. 


  Since $Q_{\Delta(e)}$ is isomorphic to $Q_{\Delta(e')}$, there exist
  some $v_{\beta}$ in $Q_{\Delta(e')}$ with an arrow of colour $c$ to
  some vertex $v_{\gamma}$, and such that there exist an isomorphism
  $Q_{\Delta(e)} \rightarrow Q_{\Delta(e')}$ that sends $v_{\alpha}$
  to $v_{\gamma}$ and $v_{e}$ to $v_{\beta}$. By Lemma
  \ref{lemconnected} we know that $\beta$ is close to the border,
  since $e$ is close to the border. Then clearly $$Q_{\Delta} \simeq
  Q_{\Delta(e)/e} \simeq Q_{\Delta(e)}/v_e \simeq
  Q_{\Delta(e')}/v_{\beta} \simeq Q_{\Delta(e')/\beta},$$   
  so hence, by assumption, $\Delta(e')/\beta = r^i \Delta$ for some
  integer $i$ with $\gamma = r^i \alpha$. We have that $\gamma$
  divides $\Delta(e')/\beta$ into two parts $A'$ and $B'$, and by
  Lemma \ref{uniqueness}, we have, say $A' = r^i A$ and $B' = r^i B$.  

  If $\beta$ is in $A'$, then $\beta = r^i e$, and we
  are finished, since $\Delta(e')/\beta = r^i (\Delta(e)/e)$. If
  not, then $s=t$, since $B' = r^j A$ for some integer $j$, and so
  hence $\alpha$ and $\gamma$ are both diameters in $P_{n,m}$,
  i.e. they connect two opposite vertices in the polygon. In any
  $(m+2)$-angulation, only one $m$-diagonal can divide the polygon
  into two parts with the same number of $m$-diagonals. But then $j =
  i+(mn+2)/2$, i.e. $\gamma = r^{i+(mn+2)/2} \alpha$ with $B' =
  r^{i+(mn+2)/2} A$ and $A' = r^{i+(mn+2)/2} B$.  

  There is only one way to extend $\Delta(e')/\beta$ in $B'$ such that
  the new vertex has an arrow of colour $c$ to
  $v_{\gamma}=v_{\alpha}$, so $\beta = r^{i+(mn+2)/2} e$. Consequently
  $r^{i+(mn+2)/2}\Delta(e) = \Delta(e'),$ and we are done.

\end{proof}

Now we can prove the following theorem by induction. 

\begin{thm} The function $$\widetilde{\sigma}_{n,m}:
  (\mathcal{T}_{n,m}/\!\!\sim) \rightarrow \mathcal{M}_{n-1,m}$$ is
  bijective for all $n$ and $m$. 
\end{thm}
\begin{proof} We already know that the function is surjective. Let $m$
  be a positive integer and suppose $\widetilde{\sigma}_{n,m}(\Delta)
  = \widetilde{\sigma}_{n,m}(\Delta')$. We want to show that $\Delta =
  \Delta'$ in $(\mathcal{T}_{n,m}/\!\!\sim)$. It is straightforward to
  check that $$\widetilde{\sigma}_{3,m}: (\mathcal{T}_{3,m}/\!\!\sim)
  \rightarrow \mathcal{M}_{2,m}$$ is injective for all $m$. 

  Suppose that $$\widetilde{\sigma}_{n-1,m}:
  (\mathcal{T}_{n-1,m}/\!\!\sim) \rightarrow \mathcal{M}_{n-2,m}$$ is
  injective for all $m$. Let $\alpha$ be an $m$-diagonal close to the
  border in $\Delta$ with image $v_{\alpha}$ in $Q$, where $Q$ is a
  representative for $\widetilde{\sigma}_{n,m}(\Delta)$. Then the
  diagonal $\alpha'$ in $\Delta'$ corresponding to $v_{\alpha}$ in $Q$
  is also close to the border. We have that
  $\widetilde{\sigma}_{n-1,m}(\Delta) =
  \widetilde{\sigma}_{n-1,m}(\Delta')$, and so by hypothesis $\Delta /
  \alpha = \Delta' / \alpha'$ in $(\mathcal{T}_{n,m}/\!\!\sim)$ 

  We can obtain $\Delta$ and $\Delta'$ from $\Delta / \alpha$ and
  $\Delta' / \alpha'$ by extending the polygon at some border edge. By
  Lemma \ref{extensionnoniso} all possible extensions of
  $\Delta/\alpha$ and $\Delta' / \alpha'$ gives non-isomorphic
  quivers, unless $\Delta = \Delta'$ in
  $(\mathcal{T}_{n,m}/\!\!\sim)$, and that finishes the proof. 
\end{proof}

\begin{cor}\label{correspondence} The number of elements in the
  $m$-mutation class of coloured quivers of type $A_{n-1}$ is the
  number of $(m+2)$-angulations of $P_{n,m}$, where two
  $(m+2)$-angulations $\Delta$ and $\Delta'$ are equivalent if $\Delta'$
  can be obtained from $\Delta$ by rotation.  
\end{cor}

We see that this agrees with the results obtained in \cite{to1} for
$m=1$, where the number of quivers in the mutation class of $A_{n-1}$
was shown to be the number of triangulations of the disk with $n$
diagonals, i.e. the number of triangulations up to rotation of a
regular polygon with $n+2$ vertices.

\section{The number of quivers in the $m$-mutation class of $A$ and
  the cell-growth problem} 

In this section we want to give formulas for the number of quivers in
the $m$-mutation classes of coloured quivers of Dynkin type $A$. 

It is easy to compute the number of non-isomorphic indecomposable
objects in $\mathcal{C}_H^m$, and we just give the formula without
proof. The number of indecomposable objects in $\mathcal{C}_H^m$ of
Dynkin type $A_n$ is given by $$\frac{mn(n+1)+2n}{2}.$$ 

It is known from \cite{z} that there is a bijection between the set of
$m$-clusters as defined in \cite{fr} and the set of $m$-cluster tilting
objects. The $m$-cluster tilting objects are in $1-1$ correspondence
with $(m+2)$-angulations, but as we have seen two tilting objects may give
rise to the same cluster-tilted algebra and the same coloured
quiver. The number of $m$-cluster tilting objects is the number of
$(m+2)$-angulations, where two $(m+2)$-angulations are considered
distinct if they are rotations of each other. In the case $m=1$ this
is the number of triangulations of a polygon, and is known to be
the Catalan numbers. This also agrees with \cite{fz2}, which says that
the number of clusters of Dynkin type $A$ is the Catalan numbers. 

The number of $m$-clusters of type $A$ is known from \cite{fr}. These
numbers are called the Fuss-Catalan numbers of type $A$. Fuss proved
already in 1791 that the number of $(m+2)$-angulations of an
$(mn+2)$-gon is given by these numbers \cite{f}. See \cite{a} for a
historical reference and more details. 

The number of $(m+2)$-angulations of a polygon is related to the
cell-growth problem, which has been investigated by many authors. We
will use the results obtained in \cite{hpr} by Harary, Palmer and
Read in 1975. 

The cell-growth problem considers an ``animal'' that grows from a cell
with a certain shape, i.e number of sides. See Figure \ref{figanimals} for some
examples of such ``animals''. The problem considered in \cite{hpr} is to
count the number of such ``animals'' with a fixed number of
cells. They also consider ``animals'' where two are equivalent if one
can be brought into the other by rotations and reflections. 

  \begin{figure}[htp]
  \begin{center}
    \includegraphics[width=4cm]{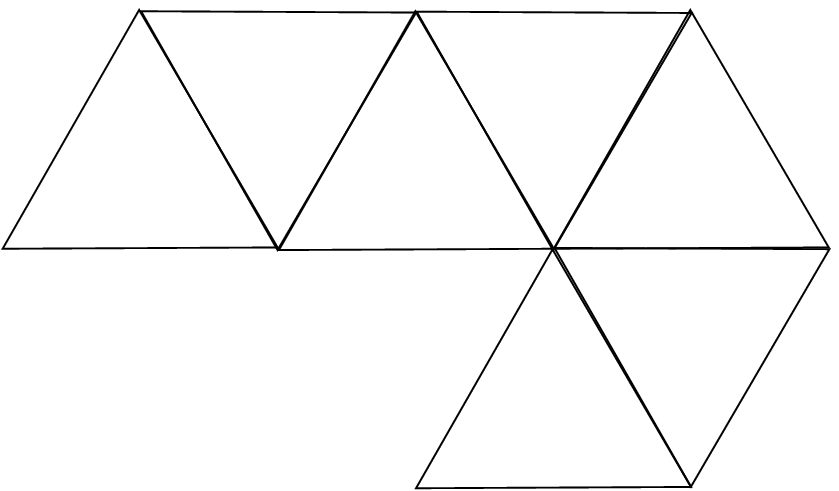}
    \includegraphics[width=3cm]{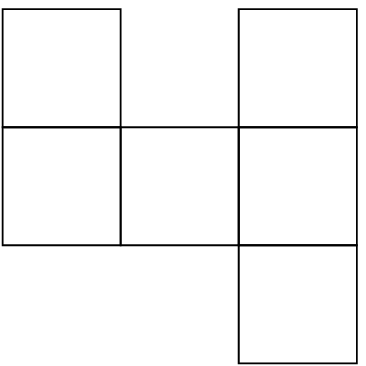}
    \includegraphics[width=4cm]{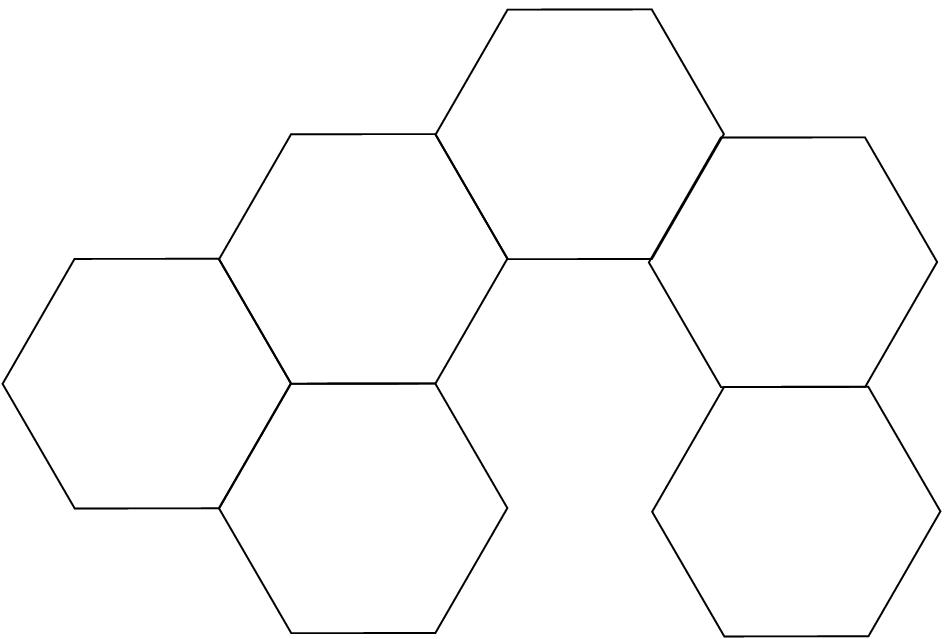}
  \end{center}\caption{\label{figanimals} Examples of ``animals'' with
    cells the shapes of polygons with 3,4, and 6 vertices.}  
  \end{figure}

In \cite{hpr} they consider ``animals'' that have a certain tree-like
structure. We refer to the paper for details. We can draw a node
inside each cell and an edge between two nodes if and only if they are
adjacent. This should result in a connected tree and not a graph. See
Figure \ref{figtreelikestructure}. 

  \begin{figure}[htp]
  \begin{center}
    \includegraphics[width=5cm]{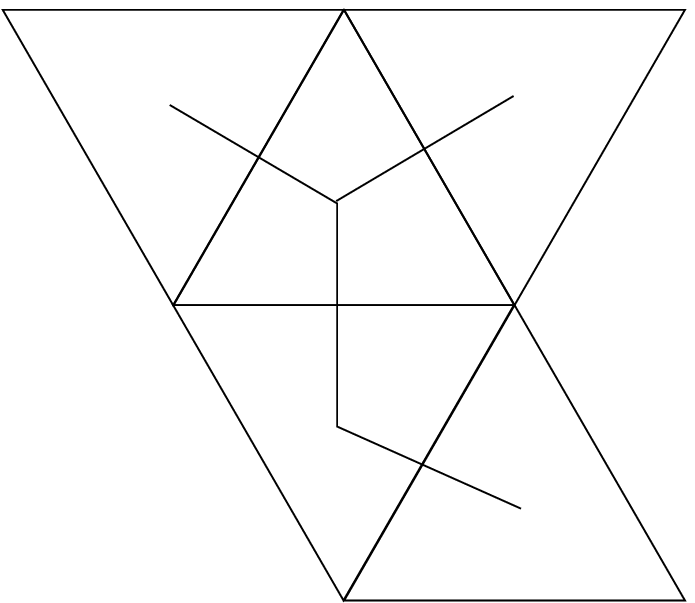}
    \includegraphics[width=5cm]{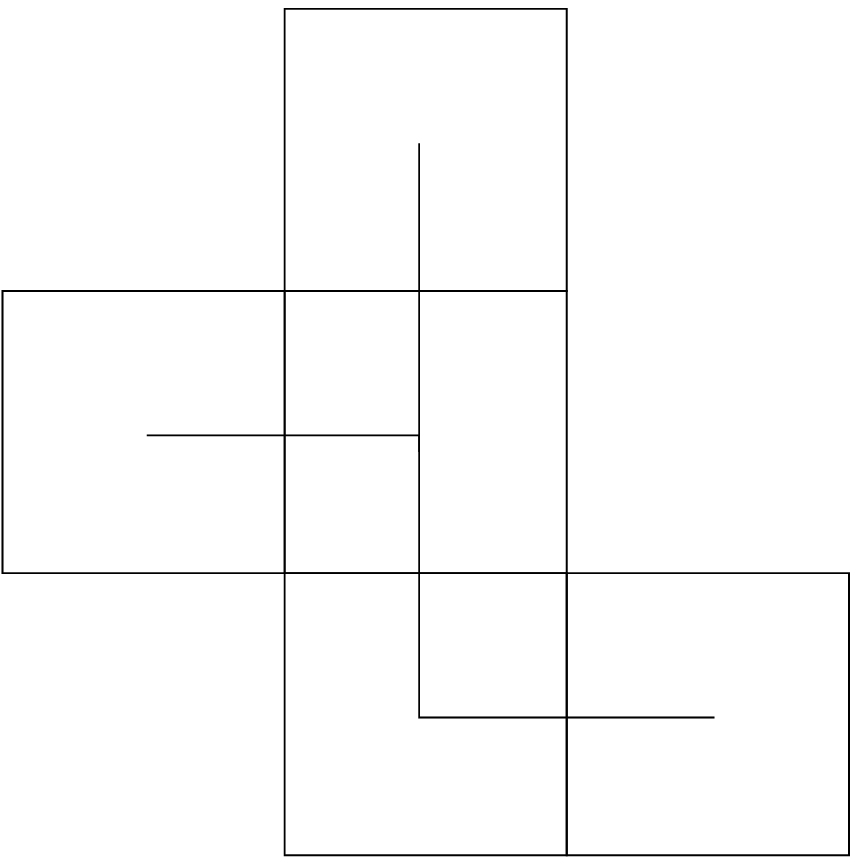}\newline

    \includegraphics[width=5cm]{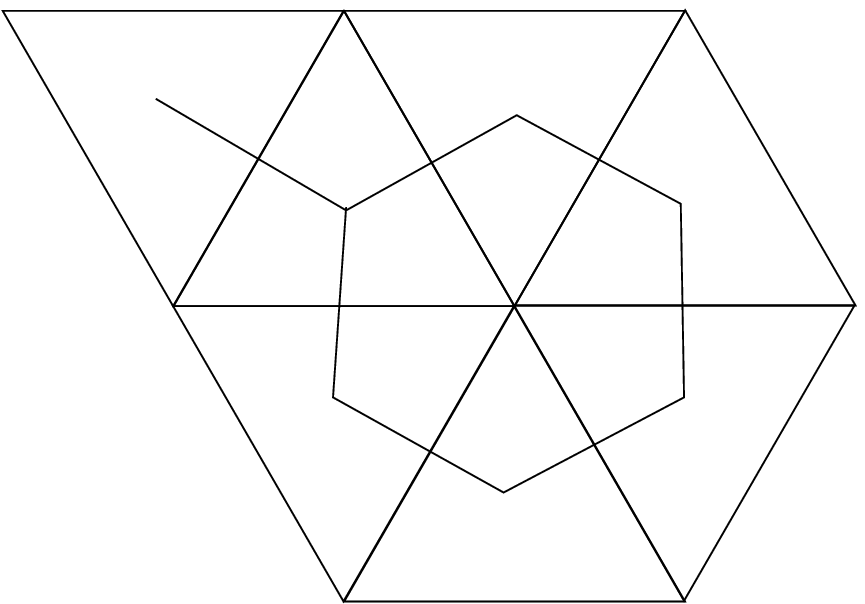}
    \includegraphics[width=5cm]{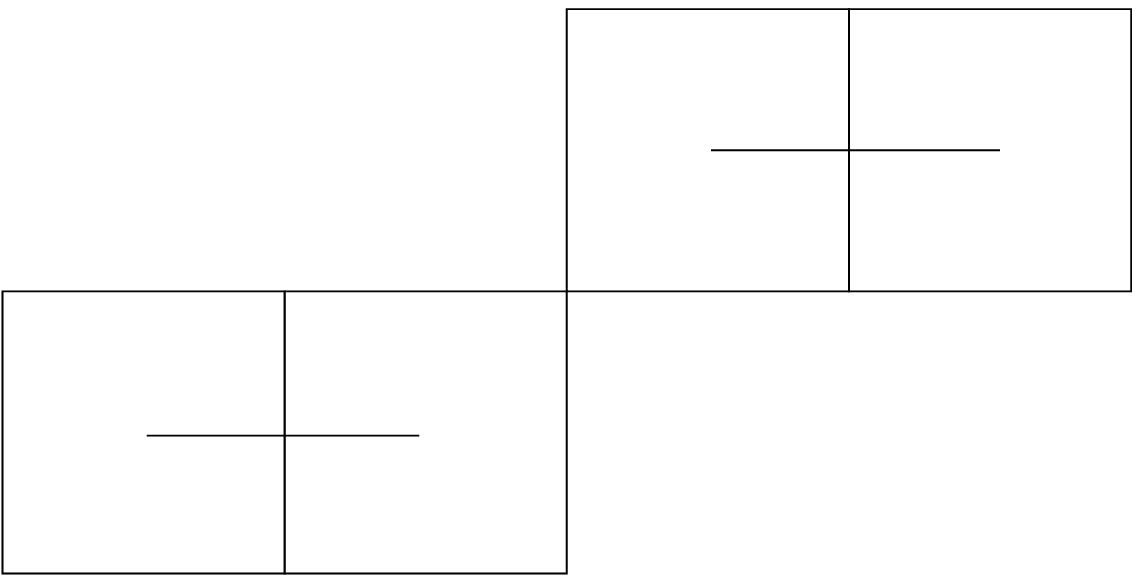}
  \end{center}\caption{\label{figtreelikestructure} The two first
    pictures are ``animals'' with allowed tree-like structure, while
    the third and fourth are not allowed.} 
  \end{figure}

The cells have the shape of a polygon with $s$ vertices, and an
``animal'' with allowed tree-like structure is called an
$s$-cluster. This is a rather fortunate name, because we will see that
such clusters or ``animals'' are in $1-1$ correspondence with
cluster-tilting objects, which again are in $1-1$ correspondence with
clusters in the cluster algebra sense. We will try not to confuse the
two different meanings in the rest of this paper.    

It is well-known that counting the number of $s$-clusters is the same
as counting the number of $s$-angulations of polygons. This follows
because of the tree-structure of the clusters. Given an $s$-angulation of
a polygon, we can construct an $s$-cluster, and given an $s$-cluster
we can construct an $s$-angulation. See Figure
\ref{figclusterangulation} for some examples. To construct a cluster,
we dissect the polygon into $s$-gons given by its $s$-angulation, and
we let two cells be adjacent if they were adjacent in the polygon. To
construct an $s$-angulation, we simply let the outer edges of the
cluster be the border of the polygon.  

  \begin{figure}[htp]
  \begin{center}
    \includegraphics[width=9cm]{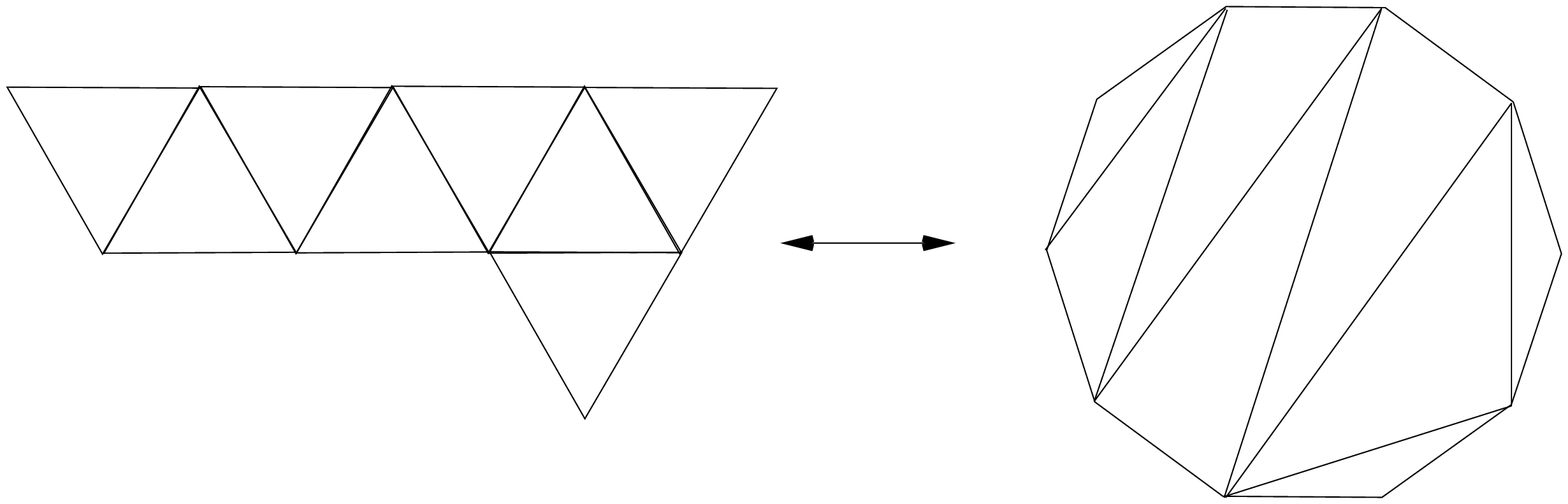}
    \includegraphics[width=9cm]{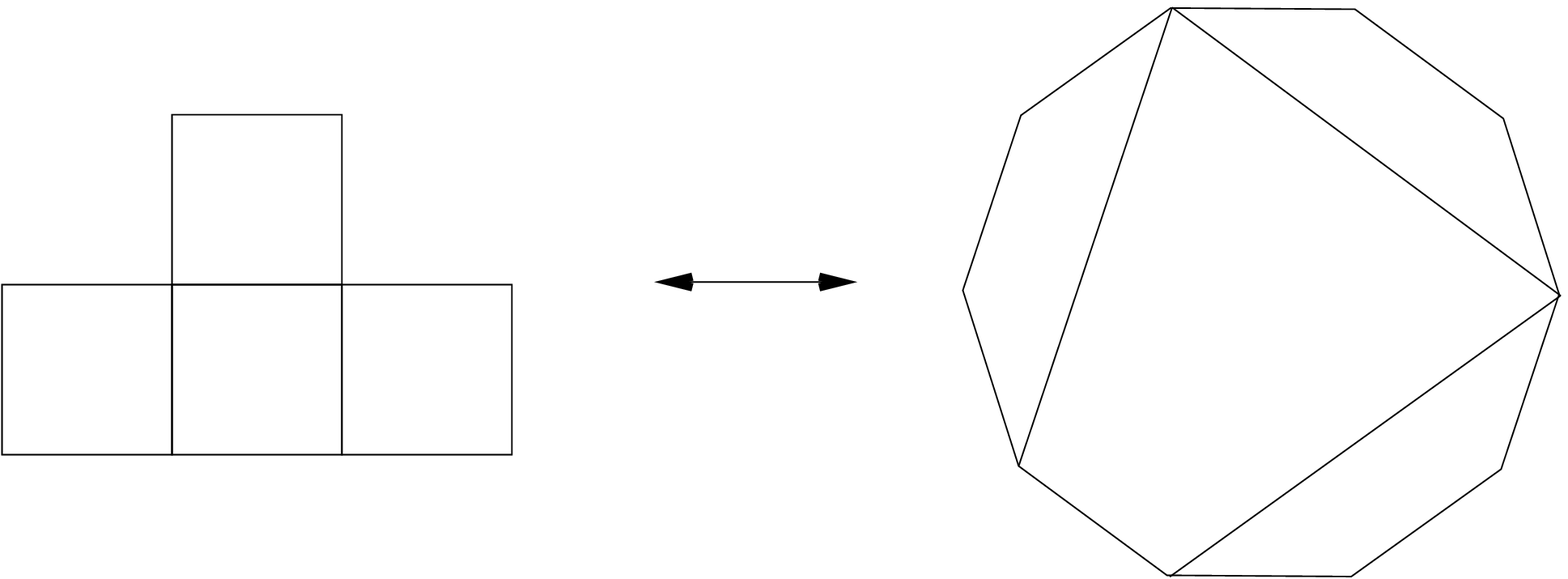}
    \includegraphics[width=9cm]{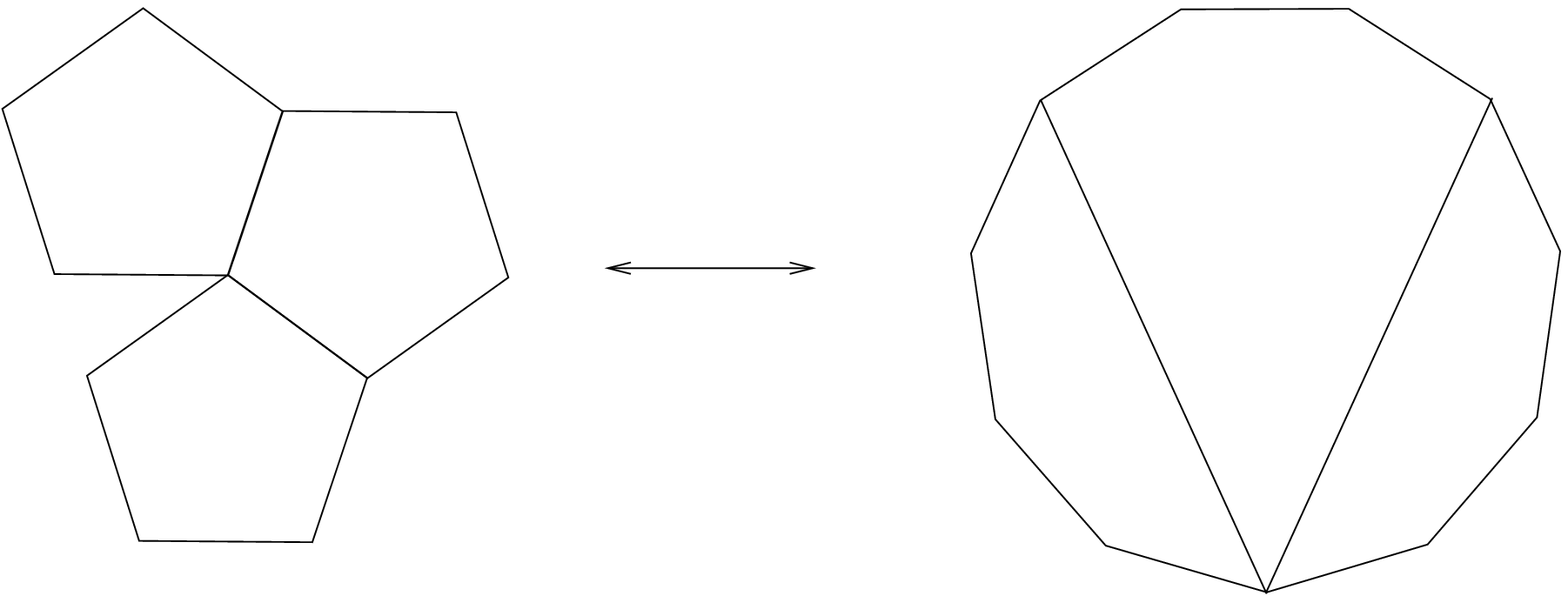}
  \end{center}\caption{\label{figclusterangulation} Correspondence between
    $s$-clusters and $s$-angulations.}  
  \end{figure}

First we consider the number of $(m+2)$-angulations of $P_{n,m}$, where
we fix the polygon in the plane and consider two $(m+2)$-angulations to
be distinct if one can be rotated into the other. We use the same
notation as in \cite{hpr}. Let

$$U_s(x) = \sum_{r=1}^{\infty} U_r^{(s)}x^r,$$
where $U_r^{(s)}$ denotes the number of $s$-clusters of $r$ cells,
rooted at an outer edge. We have the following equality \cite{hpr,p}.

$$U_s(x) = x(1+U_s(x))^{n-1}$$

By an application of Lagrange's inversion theorem they obtain that

$$U_r^{(s)} = \frac{1}{r}\binom{r(s-1)}{r-1}.$$
By the above this also gives the number of $s$-angulations of a
polygon with $r(s-2)+2$ vertices.

We obtain the following directly, and these are the Fuss-Catalan
numbers mentioned in the beginning of this section. Note that this is
a direct consequence of the correspondence in \cite{z} in combination
with the formula of \cite{fr}, see also \cite{a}. 

\begin{thm}[\cite{z,fr}]
The number of $m$-cluster tilting objects in the $m$-cluster category
$\mathcal{C}_H^m$ of type $A_{n}$ is given by 

$$U_{n+1}^{(m+2)} = \frac{1}{n+1}\binom{(n+1)(m+1)}{n}.$$
\end{thm}
\begin{proof}
This follows from the above discussion and the fact that
$(m+2)$-angulations are in $1-1$ correspondence with $m$-cluster
tilting objects.
\end{proof}
Observe that these numbers give the Catalan numbers when $m = 1$.

Next, we want to count the number of coloured quivers in the
$m$-mutation class, and by Corollary \ref{correspondence} we consider
$(m+2)$-angulations that are equivalent under rotation. In \cite{hpr}
they consider $s$-clusters rooted at a cell, and by applying P\`olyas
Theorem they obtain the series 

$$F_s(x) = \sum_{r=1}^{\infty} F_r^{(s)} x^r = xZ(C_s;1+U_s(x)),$$
where $F_r^s$ gives the number of $s$-clusters with $r$ cells rooted
at a cell and $C_s$ is the cyclic group of order $s$.

Next they consider unrooted clusters, and by applying P\`olyas Theorem
again they obtain give the generating function $$H_s(x) = F_s(x)
-\frac{1}{2}(U^2_s(x) - U_s(x^2)).$$ This counts the number of
unrooted $s$-clusters, and so hence $s$-angulations of polygons, where
two $s$-angulations are equivalent if they can be rotated into one another. 

Now we want to compute the coefficients to obtain explicit formulas,
which they do not do in \cite{hpr}.

First, we consider $U^t_s(x)$ for some $t$. We apply Lagrange's
inversion theorem (see for example \cite{hp}, Chapter 1.7) on the equation

$$U_s(x) = x(1+U_s(x))^{n-1}.$$

Set $y=U_s(x)$, and we get

$$y^t = \sum_{i=1}^{\infty}
\frac{x^i}{i!}\left(\left(\frac{d}{dy}\right)^{i-1}\left(y^{t-1}(1+y)^{(n-1)i}\right)\right)_{y=0}.$$

Differentiating and substituting we get
$$U_s^t(x)=
\sum_{i=t}^{\infty}\frac{t}{i}\binom{i(s-1)}{i-t}x^i.$$
We see that this agrees with the expression for $U_s(x)$ for $t=1$.

Now we compute $F_s(x)$. We get

$$F_s(x) = x \text{Z}(\text{C}_s;1+U_s(x))=$$

$$=x\frac{1}{s}\sum_{d|s} \phi(d)(1+U_s(x^d))^{s/d}=$$

$$= \frac{1}{s}\sum_{d|s}\phi(d) x
\sum_{t=0}^{s/d}\binom{n/d}{t}U_s(x^d)^t=$$

$$= \frac{1}{s}\sum_{d|s}\phi(d) x
\sum_{t=0}^{s/d}\binom{n/d}{t}\sum_{i=t}^{\infty}\frac{t}{i}\binom{i(s-1)}{i-t}x^{di}=$$ 

$$= \frac{1}{s}\sum_{d|s}\phi(d) x \left(1 +
\sum_{t=1}^{s/d}\binom{n/d}{t}\sum_{i=t}^{\infty}\frac{t}{i}\binom{i(s-1)}{i-t}x^{di}\right)=$$   

$$= \frac{1}{s}\sum_{d|s}\phi(d) \left(x +
\sum_{t=1}^{s/d}\binom{n/d}{t}\sum_{i=t}^{\infty}\frac{t}{i}\binom{i(s-1)}{i-t}x^{di+1}\right)$$

$$= \frac{1}{s}\sum_{d|s}\phi(d) x +
\frac{1}{s}\sum_{d|s}\phi(d)
\sum_{t=1}^{s/d}\binom{n/d}{t}\sum_{i=t}^{\infty}\frac{t}{i}\binom{i(s-1)}{i-t}x^{di+1}=$$

$$= x +
\frac{1}{s}\sum_{d|s} \sum_{t=1}^{s/d} \sum_{i=t}^{\infty}\phi(d)\frac{t}{i}
\binom{n/d}{t}\binom{i(s-1)}{i-t}x^{di+1}.$$   

We are interested in the coefficients of this series. Clearly the
coefficient of $x$ is always $1$. Suppose we want
the coefficient of $x^k$. Then $di+1 = k$ and so $i =
\frac{k-1}{d}$. This means that we want to sum over all $d$ that also 
divides $k-1$. Also we want to sum over all $t$ which give a
contribution, that is up to $\min\{\frac{s}{d},\frac{k-1}{d}\}$. We get that
the coefficient of $x^k$ is given by

$$=\sum_{d|s \text{ \& } d|(k-1)} \sum_{t=1}^{\min\{\frac{s}{d},\frac{k-1}{d}\}}
\left[\frac{\phi(d)dt}{s(k-1)}\binom{s/d}{t}\binom{(s-1)\frac{k-1}{d}}{\frac{k-1}{d}-t}
\right].$$ This formula counts the number of clusters rooted at a cell.

Now we need to compute the coefficients of $$\frac{1}{2}(U^2_s(x) -
U_s(x^2)).$$

We get
$$\sum_{i=2}^{\infty} \frac{1}{i}\binom{(s-1)i}{i-2}x^i -
\sum_{j=1}^{\infty}\frac{1}{2j}\binom{(s-1)j}{j-1}x^{2j}.$$

The coefficient of $x^k$ is given by

$$\frac{1}{k}\binom{(s-1)k}{k-2} -
\frac{1}{k}\binom{(s-1)(k/2)}{(k/2)-1},$$
where the last term is omitted if $2 \not| k$.

We obtain the following theorem, which follows from the above
discussion and Corollary \ref{correspondence}. See some examples in
Figure \ref{fignumbers}. 

\begin{thm} The number of non-isomorphic coloured quivers in the
  $m$-mutation class of a quiver of Dynkin type $A_{n}$ is given by
$$\sum_{d|n \text{ $\mathrm{\&}$ }d|m+2}\sum_{t=1}^{\min\{\frac{m+2}{d},\frac{n}{d}\}}
\left[\frac{\phi(d)dt}{(m+2)n}\binom{(m+2)/d}{t}\binom{(m+1)\frac{n}{d}}{\frac{n}{d}-t}
\right]-$$ 
$$\frac{1}{n+1}\binom{(m+1)(n+1)}{n-1} +
\frac{1}{n+1}\binom{(m+1)(n+1)/2}{(n+1)/2-1},$$
where the last term is omitted if $2 \not| n+1$. 
\end{thm}

From this formula we can obtain the formula given in \cite{to1} for
the number of elements in the mutation class of $A_n$. In this case $m
= 1$, and $d$ takes the value $1$ and the value $3$ if $3 | n$. After
some calculations we obtain the formula,
namely $$C(n+1)/(n+3)+C((n+1)/2)/2+(2/3)C(n/3),$$ where $C(i)$ is the
$i$'th Catalan number and the second term is omitted if $(n+1)/2$ is
not an integer and the third term is omitted if $n/3$ is not an
integer.

\small
\begin{center}
\begin{figure}
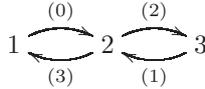

\begin{tabular}{|l|l|l|l|l|}
& m=1&m=2&m=3&m=4\\
\hline
n=2&1&2&2&3\\
n=3&4&7&12&19\\
n=4&6&25&57&118\\
n=5&19&108&366&931\\
n=6&49&492&2340&7756\\
n=7&150&2431&16252&68685\\
n=8&442&12371&115940&630465\\
n=9&1424&65169&854981&5966610\\
n=10&4522&350792&6444826&57805410\\
n=11&14924&1926372&49554420&571178751\\
n=12&49536&10744924&387203390&5737638778\\
n=13&167367&60762760&3068067060&58455577800\\
n=14&570285&347653944&24604111560&602859152496\\
n=15&1965058&2009690895&199398960212&6283968796705\\
n=16&6823410&11723100775&1631041938108&66119469155523\\
n=17&23884366&68937782355&13451978877748&701526880303315\\
n=18&84155478&408323229930&111765327780200&7498841128986110\\
n=19&298377508&2434289046255&934774244822704&80696081185767000\\
n=20&1063750740&14598011263089&7865200653146910&873654669882575000\\

\end{tabular}\caption{\label{fignumbers}Some numbers of coloured quivers in
  the mutation classes for various values of $n$ and $m$.}  
\end{figure}
\end{center}
\normalsize

\section{Further remarks}

In \cite{m} the author gives a description of the Gabriel quivers in
the $m$-mutation class of quivers of Dynkin type $A$ together with the
relations. Let $Q$ be a quiver in the $m$-mutation class of $A_n$, and
let $\Delta$ be a corresponding $(m+2)$-angulation. Then $$i
\rightarrow j \rightarrow k$$ is a zero path if and only if the
corresponding $m$ diagonals of $i$, $j$ and $k$ are in the same
$(m+2)$-gon in $\Delta$. It is also shown that there are no
commutativity relations.  

Let $\Delta$ and $\Delta'$ be two $(m+2)$-angulations, and
$Q_{\Delta}$ and $Q_{\Delta'}$ the corresponding coloured quivers. We
know that $Q_{\Delta}$ and $Q_{\Delta'}$ are isomorphic if and only if
$\Delta$ and $\Delta'$ are rotations of each other. However, if
$Q_{\Delta}$ and $Q_{\Delta'}$ are not isomorphic as coloured quivers,
the Gabriel quivers may still be isomorphic. Also, two non-isomorphic
coloured quivers can give rise to the same $m$-cluster tilted
algebra. For example, if $Q$ is the quiver 

$$\xymatrix{1 \ar@/^/[r]^{(0)} & 2 \ar@/^/[l]^{(3)}\ar@/^/[r]^{(2)} &
  3\ar@/^/[l]^{(1)}}$$ 

occurring for a $3$-cluster tilted algebra, then the quiver $Q'$ 

$$\xymatrix{1 \ar@/^/[r]^{(0)} & 2 \ar@/^/[l]^{(3)}\ar@/^/[r]^{(1)} &
  3\ar@/^/[l]^{(2)}}$$ 
obtained from $Q$ by mutating at $3$, is not isomorphic to $Q$ as
coloured quivers. However, they give the same $3$-cluster tilted
algebra.



\textbf{Acknowledgements:} The author would like to thank Aslak Bakke
Buan for valuable discussions and comments.

\end{document}